\documentclass[12pt]{amsart}
\usepackage{amsmath, amssymb, amsthm}
\usepackage{graphicx}
\usepackage[margin=1in]{geometry}

\newcommand{\R}{\mathbb{R}}

\theoremstyle{plain}
\newtheorem{thm}{Theorem}[section]
\newtheorem*{thm*}{Theorem}

\theoremstyle{definition}

\theoremstyle{remark}
\newtheorem*{rem}{Remark}

\begin{document}

\author{Dami Lee}
\title{On a triply periodic polyhedral surface whose vertices are Weierstrass points}
\address{Department of Mathematics, Indiana University, Bloomington, IN 47405}
\email{damilee\char`\@indiana.edu}
\urladdr{}
\date{}

\maketitle

\begin{abstract}
In this paper, we will construct an example of a closed Riemann surface $X$ that can be realized as a quotient of a triply periodic polyhedral surface $\Pi \subset \R^3$ where the Weierstrass points of $X$ coincide with the vertices of $\Pi.$ First we construct $\Pi$ by attaching Platonic solids in a periodic manner and consider the surface of this solid. Due to periodicity we can find a compact quotient of this surface, which has genus $g = 3.$ We claim that the resulting surface is regular in the hyperbolic sense. By regular, we mean that the automorphism group of $X$ is transitive on flags. The symmetries of $X$ allow us to construct hyperbolic structures and various translation structures on $X$ that are compatible with its conformal type. The translation structures are the geometric representations of the holomorphic 1-forms of $X,$ which allow us to identify the Weierstrass points.
\end{abstract}

\section{Introduction}
In this paper we generalize the construction of triply periodic polyhedral surfaces by Coxeter and Petrie \cite{C}. They introduced three triply periodic regular polyhedra whose quotient surface by euclidean translations is a surface of genus $g = 3.$ In Schlafli symbols, they are denoted as $\{4, 6 | 4\}, \{6, 4 | 4\}, \{6, 6 | 3\}$ where $\{p, q | r\}$ represents a regular polyhedron that is constructed by $q$ regular $p$-gons at each vertex forming regular $r$-gonal holes. The polyhedral structures induce cone metrics. Furthermore the cone metrics induce a conformal structure on the underlying Riemann surface. These three surfaces are related to minimal surfaces in the sense that they are polygonal approximations of well known minimal surfaces, that also carry the same conformal structure. $\{4, 6 | 4\}$ and its dual $\{6, 4 | 4\}$ serve as the approximation of Schwarz's P-surface and $\{6, 6 | 3\}$ as Schwarz's D-surface. Moreover for instance the vertices of $\{4, 6 | 4\}$ coincide with the Weierstrass points of the surface that is conformally equivalent.\\

\begin{figure}[htbp] 
\centering
\begin{minipage}{.5\textwidth}
	\centering
	\includegraphics[width=2in]{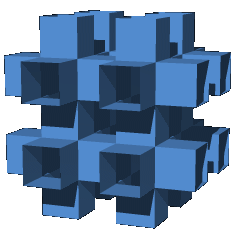}
	\end{minipage}
\begin{minipage}{.5\textwidth}
	\centering
	\includegraphics[width=2in]{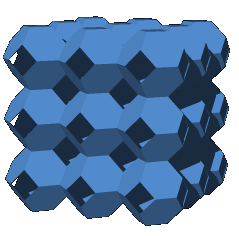}
	\end{minipage}
\begin{minipage}{.5\textwidth}
	\centering
	\includegraphics[width=2in]{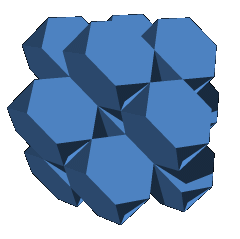}
	\label{coxeter}
	\end{minipage}
	\caption{Infinite regular skew polyhedra in 3-space $\{4,6|4\}, \{6,4|4\},$ and $\{6,6|3\}.$ Reprinted from Wikipedia, The Free Encyclopedia, by Tom Ruen, Retrieved from 
$<\text{https://en.wikipedia.org/wiki/Regular\underline{\hspace{0.1in}}polyhedron.}>$}

\end{figure}

This paper introduces an example of a triply periodic polyhedral surface that is regular in a slightly weaker sense but shares many properties with the surfaces that Coxeter and Petrie introduce. Adopting the Schlafli symbols, this surface is of type $\{3, 8 | 3\}$ and when quotiented by its Euclidean translations we get a genus $g = 3$ Riemann surface. This turns out to be a cyclically branched cover over a thrice punctured sphere. This allows us to find various other cone metrics and specifically those that are translation structures. The translation structures give us enough holomorphic 1-forms to compute the Wronksi metric. For our surface, the Wronski metric happens to coincide with the cone metric. This allows us to find all Weierstrass points, identify the conformal automorphism group, and rule out that the surface is hyperelliptic. In particular, this surface cannot be conformally equivalent to a triply periodic minimal surface \cite{M}.\\

We relax the definition of regularity on the polyhedral surface to find more symmetries on the underlying abstract surface. For platonic solids and the surfaces that Coxeter and Petrie introduce, one requires the euclidean isometries to be transitive on flags which we define in Section~\ref{automorphisms}. The surface that we introduce also have regular polygonal faces and with the same valency at each vertex. However it is not flag transitive in the eucliean sense. Once we consider its hyperbolic structure we can find all hyperbolic automorphisms on the underlying Riemann surface and show that the automorphism group acts transitively on flags.\\

This surface has a nice distribution of Weierstrass points, a large automorphism group, and serves as an explicit example of a non-hyperelliptic Riemann surface whose flat structure, hyperbolic structure, and algebraic expressions are determined explicitly . This is a part of an ongoing project on relations between triply periodic polyhedral surfaces and triply periodic minimal surfaces.\\

This paper is organized as follows:\\
\begin{itemize}
\item In Section~\ref{construction}, we construct a polyhedral surface $\Pi$ and find the compact quotient $X$ of $\Pi$ by its periodicity in euclidean space.
\item In Section~\ref{structure}, we observe the hyperbolic tiling of the fundamental piece of $\Pi$ and discuss the geodesics on $X,$ which later guides us to finding hyperbolic automorphisms and translation structures.
\item In Section~\ref{automorphisms}, we determine the automorphism group of $X.$ 
\item In Section~\ref{form}, we find translation structures on $X,$ which give us holomorphic 1-forms on the surface. This gives us an algebraic description of this surface.
\item In Section~\ref{Weierstrass}, we find Weierstrass points of $X$ and show that the set of all Weierstrass points coincides with the set of all vertices on $\Pi.$
\end{itemize}

The author would like to thank Matthias Weber for his support.

\section{The construction of a triply periodic polyhedral surface}
\label{construction}

In this section we will build a triply periodic polyhedral surface $\Pi$ that embeds in euclidean space by attaching regular octahedra along faces. For instance, we may start with two regular octahedra and glue them along a pair of faces. Since every octahedron consists of four pairs of parallel faces we can continue attaching octahedra so that the new octahedra are attached in a parallel manner. This gives us a tower of octahedra one on top of each other. However, our goal is to build a triply periodic figure so instead of gluing octahedra on parallel faces, we construct our surface with two different types of octahedra. First we start with two regular octahedra of the same size. One type of octahedron (Type 1) will be a regular octahedron minus two parallel faces. The other type (Type 2) will be a regular octahedron minus four non-adjacent faces. We will glue Type 1 and Type 2 octahedra on their missing faces in an alternating order so that on every Type 1 octahedron we have two Type 2 octahedra attached and on every Type 2 octahedron we have four Type 1 octahedra attached. First we claim that the surface constructed this way embeds in $\R^3.$\\

\begin{thm} The surface $\Pi$ constructed via alternating octahedra embeds in $\R^3.$
\end{thm}

\begin{proof} The idea is to place regular octahedra in a regular cube and tile space with cubes. We let the six vertices of the octahedron sit on the edges of the cube so that they divide the edge with a 1:3 ratio. The reason for this ratio is because if we truncate the cube along the faces of the octahedron, the truncated part that is facing us has the same size as one-eighth of the octahedron inside the cube. We can view this octahedron as a Type 1 octahedron where the face facing us and its opposite face are the two missing faces. By space-filling property of cubes, we can place eight of these cubes around that truncated vertex and get Figure~\ref{octa4-2}.

\begin{figure}[htbp] 
\centering
\begin{minipage}{.5\textwidth}
	\centering
	\includegraphics[width=2in]{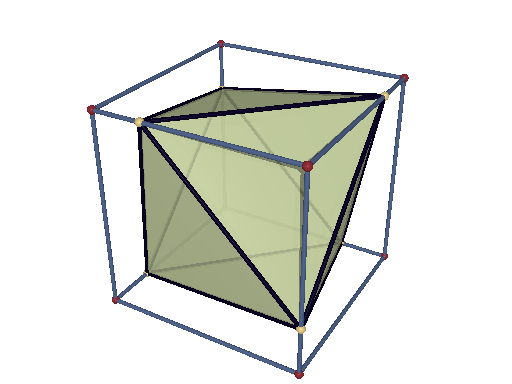}
	\caption{}
	\label{octa4-1}
\end{minipage}%
\begin{minipage}{.5\textwidth}
	\centering
	\includegraphics[width=2in]{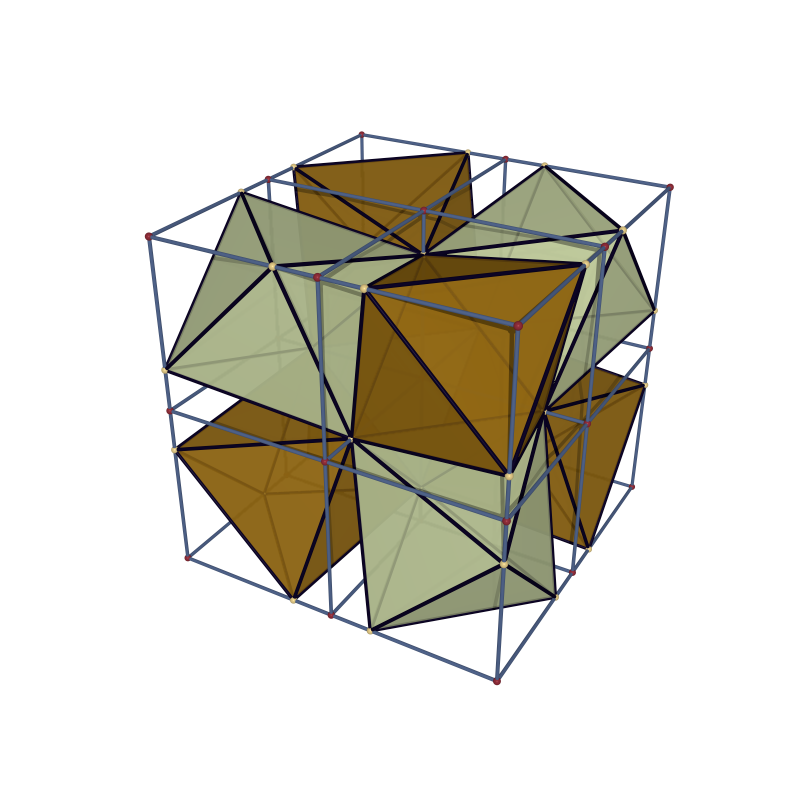}
	\caption{}
	\label{octa4-2}
\end{minipage}
\end{figure}

However for the purpose of this paper we will need only four cubes around a vertex so that no two cubes share faces with each other. We will place a Type 2 octahedron where the four cubes meet as in the following figure.

\begin{figure}[htbp] 
	\centering
	\includegraphics[width=2in]{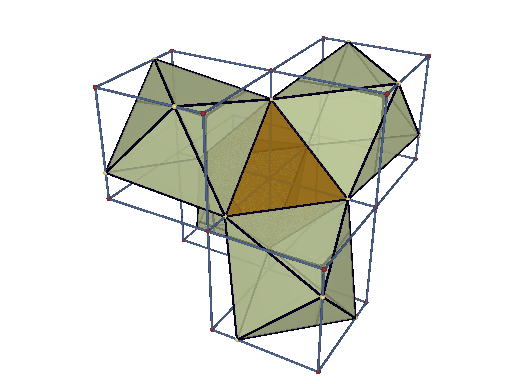} 
	\caption{}
	\label{octa4}
\end{figure}

Since Type 1 octahedra miss two parallel faces, we can continue and attach other Type 2 octahedra on the opposite sides of Type 1 octahedra. Since the cubes in which Type 1 octahedra sit in form a subtiling of $\R^3,$ our surface too embeds in $\R^3.$

\begin{figure}[htbp] 
\centering
\begin{minipage}{.5\textwidth}
	\centering
	\includegraphics[width=2.5in]{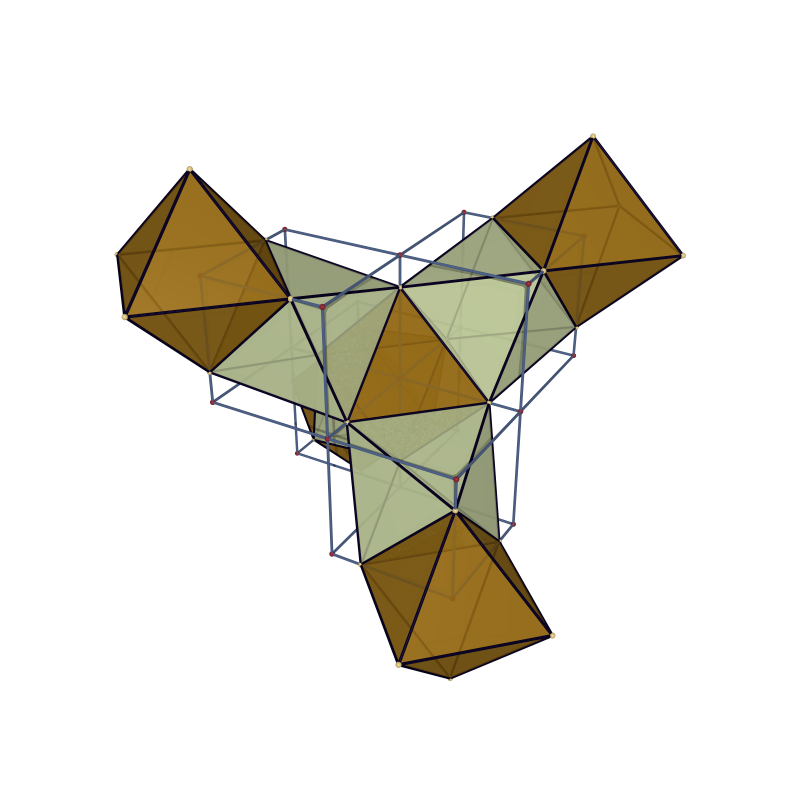}

\end{minipage}%
\begin{minipage}{.5\textwidth}
	\centering
	\includegraphics[width=2.5in]{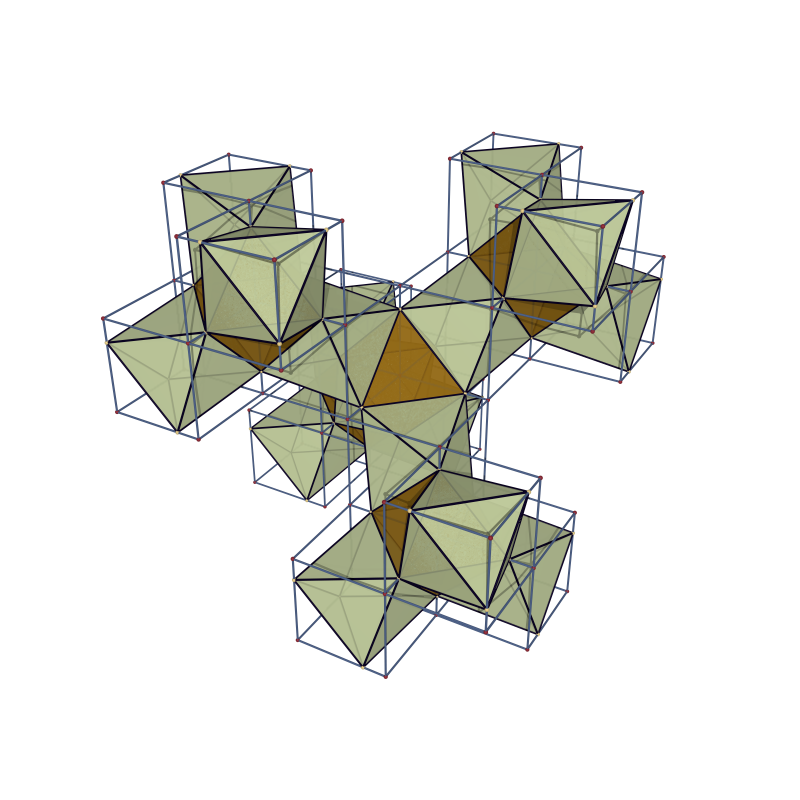}
\end{minipage}
	\caption{Construction of $\Pi$}
	\label{octa4-5}
\end{figure}

\end{proof}

Figure~\ref{octa4-5} shows us that this surface $\Pi$ is triply periodic. In other words, the surface is invariant under three independent translations. Hence we can identify the faces via translations and get a smallest piece say $X_0$ that spans $\Pi.$ Now we would like to prove that $X_0$ which we will call the fundamental piece of $\Pi$ can be presented as the following Figure~\ref{funda}.

\begin{figure}[htbp] 
	\centering
	\includegraphics[width=2.5in]{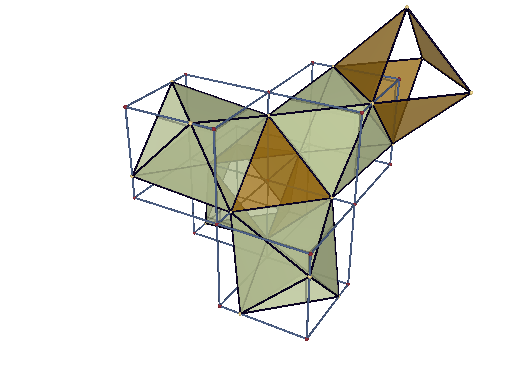} 
	\caption{Fundamental piece, $X_0$}	
	\label{funda}
\end{figure}

\begin{thm}\label{thm funda} $X_0$ is the fundamental piece of the triply periodic surface $\Pi$ which we can span via three independent parallel translations.
\end{thm}

\begin{proof} Notice that this piece has four Type 1 octahedra and two Type 2 octahedra. So in Figure~\ref{funda}, there are six faces that are removed from the regular octahedra. We will show that these six faces can be identified within themselves via parallel translation.\\

Let's start from a face that is missing from one of the Type 1 octahedra. Our goal is to find a face that can be identified with this face via parallel translation. In Figure~\ref{funda} there is a Type 2 octahedron attached to this Type 1 octahedron. However it does not have a missing face that is parallel to the face we started with so we go to the next Type 2 octahedron. The two Type 2 octahedra are parallel since they are both attached to the same Type 1 octahedron. The second Type 2 is missing a face that is parallel to the face that we started from. Therefore these two faces can be identified to each other via parallel translation.\\

Due to the order three rotational symmetry this figure has, the identification of the other four faces follows. Since Figure~\ref{funda} is the smallest piece that can be translated to construct $\Pi,$ it actually is the fundamental piece. Since this piece is topologically a sphere with six holes, identifying these in pairs gives us that this is a surface of genus $g = 3.$ \\

Alternatively, counting the number of vertices, edges, and faces also gives us the same genus where $12 - 48 + 32 = 2 - 2 g,$ hence $g = 3.$
\end{proof}

We used the fact that there is an order three rotational symmetry on $X_0$ but not all edges are similar in euclidean space. Hence the automorphism group generated by euclidean symmetries is not transitive and the surface is not regular. However, instead of considering uclidean isometries we will consider a weaker definition of regularity and consider the isometries on its abstract quotient surface. We call the abstract surface $X$ and show that this is regular. We can see from $\Pi$ that the valency at every vertex is eight hence we have hope that the automorphism group of $X$ would be at least vertex-transitive. In the following sections, we will prove that the hyperbolic structure shows that the abstract quotient surface $X$ is actually regular via its group of hyperbolic isometries.

\section{Hyperbolic structure on $X$}
\label{structure}

Our goal in this section is to see the symmetries of $X,$ not necessarily euclidean and to do so we will consider the hyperbolic structure of the fundamental piece that is compatible with its conformal type. Using the fact that eight faces meet at every vertex, we can tile the hyperbolic disk with $(\frac{\pi}{4}, \frac{\pi}{4}, \frac{\pi}{4})$ triangles. However, we know that the fundamental piece consists of 32 triangles so we can identify the 32 triangles in the following figure.

\begin{figure}[htbp] 
   \centering
   \includegraphics[width=3in]{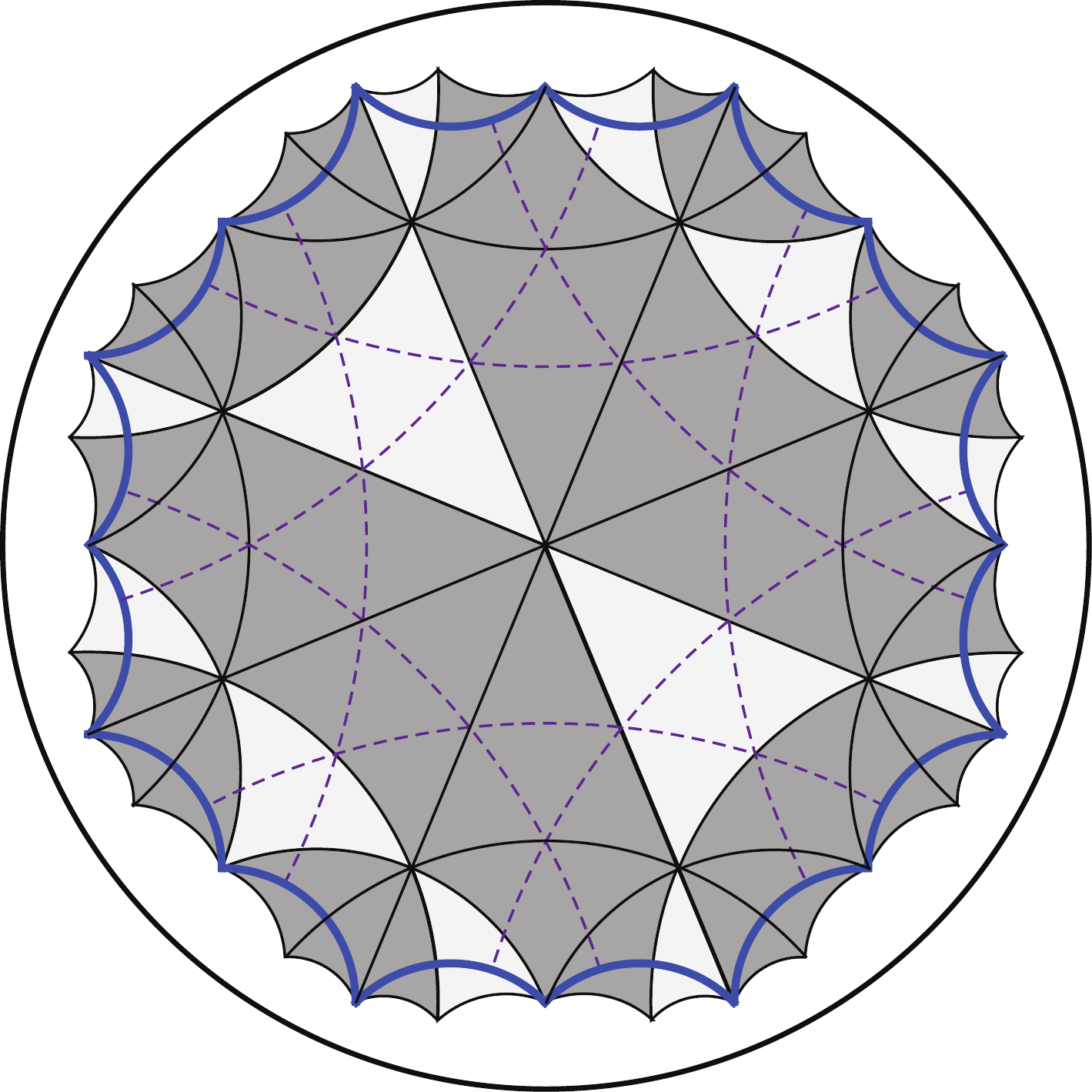} 
   \caption{Hyperbolic realization of fundamental piece}
   \label{hyperbolic_tiling}
\end{figure}

First we will discuss Petrie polygons on $\Pi$ before mentioning the geodesics on this surface and then prove that Figure~\ref{hyperbolic_tiling} represents the fundamental piece $X_0.$ By a Petrie polygon, we mean a piecewise geodesic connecting midpoints of the edges so that the ``clipped'' vertices lie alternatingly to the left and the right of the polygon. In our case, the Petrie polygons will automatically be smooth.

\begin{figure}[htbp] 
   \centering
   \includegraphics[width=3in]{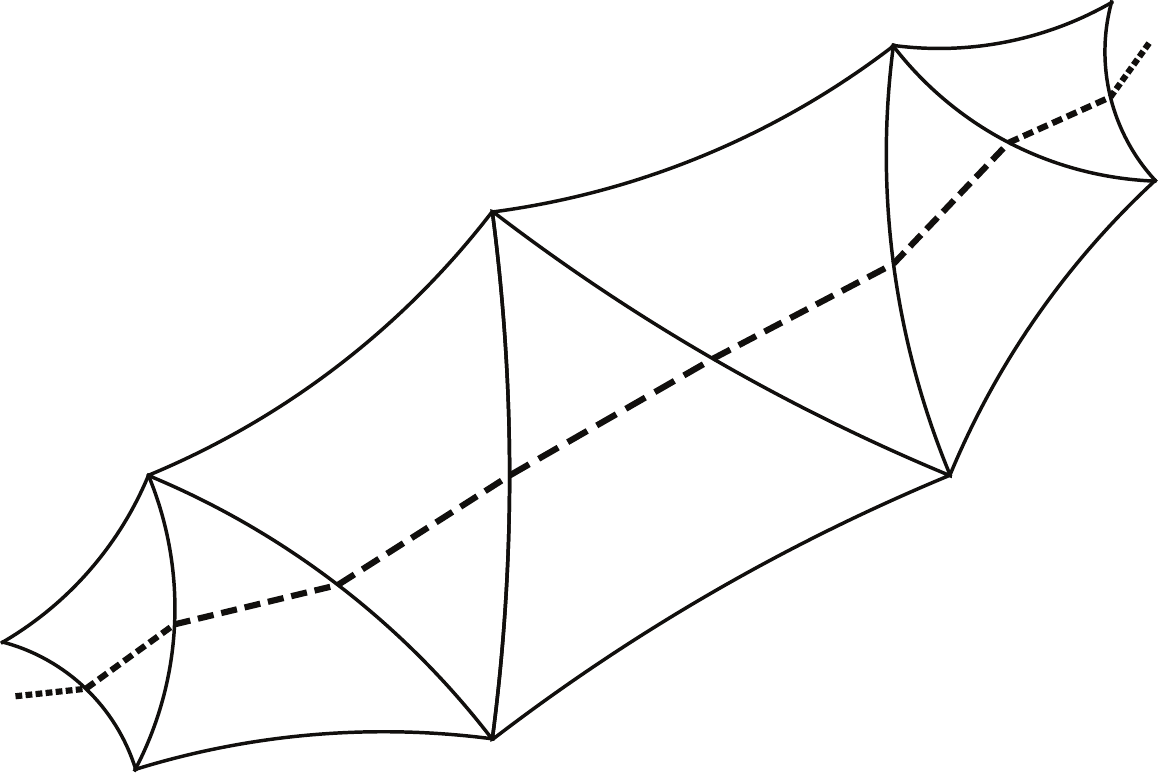} 
   \caption{Petrie polygon}
   \label{Petrie polygon}
\end{figure}

\begin{thm} All Petrie polygons on $\Pi$ correspond to a closed geodesic on $X.$ Furthermore, all geodesics of this form have the same length.
\end{thm}

\begin{proof} Since Type 1 octahedra are strips of six triangles, the Petrie polygons that remain in Type 1 octahedra form closed geodesics and have length six. For the cases that involve Type 2 octahedra, we can refer to Figure~\ref{funda} and recall the proof of Theorem~\ref{thm funda} where we identified missing faces to each other. Notice that if we pick a pair of edges in $X_0$ that are identified to each other, we can see that there is a Petrie polygon that passes six faces and connects the two edges. Hence all Petrie polygons on $\Pi$ are closed geodesics on $X$ and they all have the same length. These will be shown as dotted lines in Figure~\ref{hyperbolic_tiling}.
\end{proof}

\begin{rem} We can see in Figure~\ref{hyperbolic_tiling} that the set of geodesics is invariant under an order eight rotation centered at the center vertex.
\end{rem}

Now we prove that Figure~\ref{hyperbolic_tiling} represents the fundamental piece.\\

\begin{thm} The 16-gon that is bounded by the solid lines in Figure~\ref{hyperbolic_tiling} represents the fundamental piece. Moreover, the identification of edges by dotted lines gives us the same surface as $X.$
\end{thm}

\begin{proof} The 16-gon bounded by the solid lines in Figure~\ref{hyperbolic_tiling} consists of 32 triangles. Triangles that with darker shading come from Type 1 octahedra, where the rest come from Type 2 octahedra. Specifically, since Type 1 octahedra are strips of six triangles, they are easy to recognize. Moreover, the dotted lines that lie in these strips show us how to identify two edges of the 16-gon together. However we already know from the previous theorem that all Petrie polygons are closed whether they remain in Type 1 octahedra or not. Hence the surface that we get by the identification of edges given by dotted lines gives us the same surface as $X.$
\end{proof}

\section{Automorphisms}
\label{automorphisms}

We can see from Figure~\ref{funda} that there is an order three rotational symmetry in $X_0.$ In this section, we will use the hyperbolic structure of the fundamental piece to find automorphisms of $X,$ that are not necessarily induced from its polyhedral structure on $X_0.$ At the end of this section once we find all automorphisms of $X$ and show that $\textrm{Aut}(X)$ acts transitively on flags, we will be able to conclude that $X$ is a regular surface.\\

In this section we will use a notion of a \emph{flag} which is a triple $(v, e, f)$ where $f$ is a face, $e$ is one of the edges of $f,$ and $v$ is one of the two boundary points of $e.$ Since we are interested in orientation preserving automorphisms once we choose a face $f$ and one of its edges $e,$ we say that there is only one choice of $v.$ In other words the flags may be of the shapes $\leftharpoondown$ or $\rightharpoonup$ but we do not accept shapes $\leftharpoonup$ and $\rightharpoondown$ as flags. We will prove that we only need one order three rotation and one order eight rotation to generate the automorphim group that preserves flags on this surface.\\

In Figure~\ref{funda}, there is an order three symmetry realized as the rotation about a midpoint of a face of a Type 2 octahedron. Via this rotation, the faces of Type 1 (resp. Type 2) octahedra remained in Type 1 (resp. Type 2.) This symmetry induces also an order three rotational symmetry in $X$ that fixes a midpoint of a triangle as we can see in Figure~\ref{hyperbolic_tiling}. In fact, these symmetries map vertices to vertices, edges to edges, and faces to faces hence we call this symmetry flag-transitive.\\

Recall the remark in the previous section that there is an order eight rotational symmetry that fixed the center of the Figure~\ref{hyperbolic_tiling}. This symmetry is also flag-transitive. Notice that this is not induced from the polyhedral structure on $X_0$ but comes from the hyperbolic structure.\\

\begin{thm} Given two flags on $X,$ there exists an automorphism that sends one flag to the other. Moreover, $|\textrm{Aut}(X)| = 96.$
\end{thm}

\begin{proof} In the following figure, there is an order eight rotation that fixes the center of the tiling which we denote by $a.$ This map sends flags to flags hence the group generated by $a$ is flag-transitive. Secondly, we can denote an order three rotation by $b,$ which fixes the midpoint of one of the triangles and permutes the triangle's edges. Again $b$ preserves flags hence any group generated by $a$ and $b$ is flag-transitive. Additionally, we can think of an order two symmetry that preserves any given edge $e$ and fixes the midpoint. However $a b = b^{-1} a^{-1}$ represents such a rotation so we will only consider the group generated by two elements $a$ and $b.$

\begin{figure}[htbp] 
   \centering
   \includegraphics[width=3in]{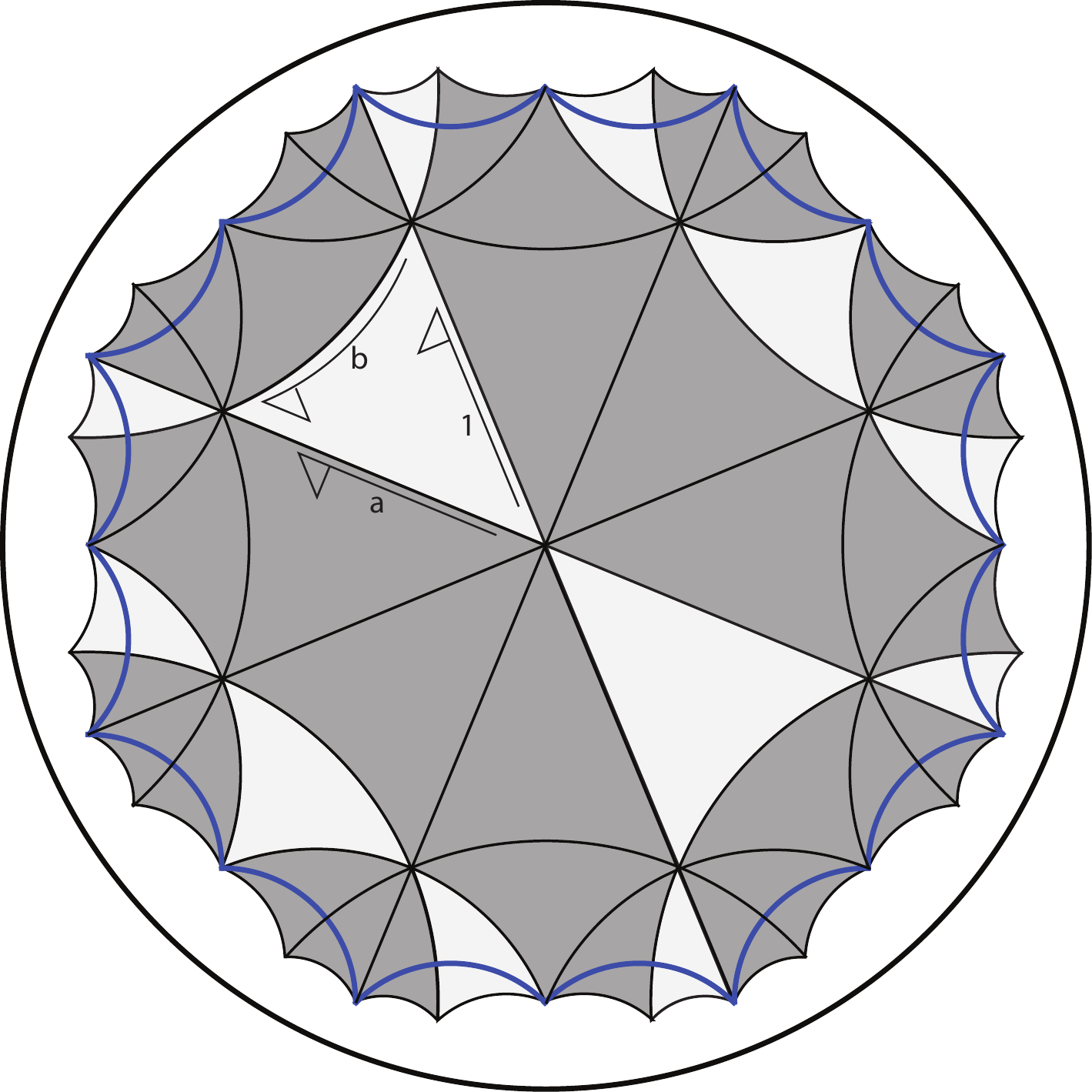} 
   \caption{Automorphisms}
   \label{aut}
\end{figure}

We can name each flag by a word generated by $a$ and $b.$ The correspondence between words and flags gives us concrete relations between generators. We get a list of 96 words which are all distinct.

\begin{table}
\begin{tabular}[t]{|c|c|c|c|c|c|c|c|}
\hline
1 &&&&&&& \\ \hline
$a$ &&&&&&& \\ \hline
$a^2$ & $a^2 b$ & $a ^2 b^2$ & $a^2 b^2 a$ &&&& \\ \hline
$a^3$ & $a^3 b$ & $a^3 b^2$ & $a^3 b^2 a$ &&&& \\ \hline
$a^4$ & $a^4 b$ & $a^4 b^2$ & $a^4 b^2 a$ & $a^4 b^2 a^2$ & $a^4 b^2 a^3$ && \\ \hline
$a^5$ & $a^5 b$ & $a^5 b^2$ & $a^5 b^2 a$ & $a^5 b^2 a^2$ & $a^5 b^2 a^3$ && \\ \hline
$a^6$ & $a^6 b$ & $a^6 b^2$ & $a^6 b^2 a$ & $a^6 b^2 a^2$ & $a^6 b^2 a^3$ & $a^6 b^2 a^4$ & $a^6 b^2 a^5$  \\ \hline
$a^7$ & $a^7 b$ &&&&&& \\ \hline
$b$ &&&&&&& \\ \hline
$b a$ &&&&&&& \\ \hline
$b a^2$ & $b a^2 b$ & $b a^2 b^2$ & $b a^2 b^2 a$ &&&& \\ \hline
$b a^3$ & $b a^3 b$ & $b a^3 b^2$ & $b a^3 b^2 a$ &&&& \\ \hline
$b a^4$ & $b a^4 b$ & $b a^4 b^2$ & $b a^4 b^2 a$ & $b a^4 b^2 a^2$ & $b a^4 b^2 a^3$ && \\ \hline
$b a^5$ & $b a^5 b$ & $b a^5 b^2$ & $b a^5 b^2 a$ & $b a^5 b^2 a^2$ & $b a^5 b^2 a^3$ && \\ \hline
$b a^6$ & $b a^6 b$ & $b a^6 b^2$ & $b a^6 b^2 a$ & $b a^6 b^2 a^2$ & $b a^6 b^2 a^3$ & $b a^6 b^2 a^4$ & $b a^6 b^2 a^5$ \\ \hline
$b a^7$ & $b a^7 b$ &&&&&& \\ \hline
$b^2$ &&&&&&& \\ \hline
$b^2 a$ &&&&&&& \\ \hline
$b^2 a^2$ & $b^2 a^2 b$ & $b^2 a^2 b^2$ & $b^2 a^2 b^2 a$ &&&& \\ \hline
$b^2 a^3$ & $b^2 a^3 b$ & $b^2 a^3 b^2$ & $b^2 a^3 b^2 a$ &&&& \\ \hline
$b^2 a^4$ & $b^2 a^4 b$ & $b^2 a^4 b^2$ & $b^2 a^4 b^2 a$ & $b^2 a^4 b^2 a^2$ & $b^2 a^4 b^2 a^3$ && \\ \hline
$b^2 a^5$ & $b^2 a^5 b$ & $b^2 a^5 b^2$ & $b^2 a^5 b^2 a$ & $b^2 a^5 b^2 a^2$ & $b^2 a^5 b^2 a^3$ && \\ \hline
$b^2 a^6$ & $b^2 a^6 b$ & $b^2 a^6 b^2$ & $b^2 a^6 b^2 a$ & $b^2 a^6 b^2 a^2$ & $b^2 a^6 b^2 a^3$ & $b^2 a^6 b^2 a^4$ & $b^2 a^6 b^2 a^5$ \\ \hline
$b^2 a^7$ & $b^2 a^7 b$ &&&&&& \\ \hline
\end{tabular}
\caption{Automorphisms generated by $a$ and $b$}
\end{table}

Since there are 32 faces and three flags on each face, there are 96 flags in total, which is equal to the number of words we have generated. We have $\textrm{Aut}(X) = \langle a, b \mid a^8 = b^3 = (a b)^2 = (a^2 b^2)^3 = (a^4 b^2)^3 = 1 \rangle$ and $\lvert \textrm{Aut}(X) \rvert = 96.$ 
\end{proof}

\begin{rem} By inspection, $X \rightarrow X / \langle a \rangle$ is a cyclically branched eightfold cover over a thriced punctured sphere $X / \langle a \rangle.$ Klein's quartic in \cite{KW} is also viewed as a sevenfold cyclically branced cover over a thriced punctured sphere.\\
\end{rem}

The automorphism group of the surface acts transitively on flags, in other words, the surface is regular. This implies that all 12 vertices are similar, which we will use in Section~\ref{Weierstrass} to prove that all vertices of $\Pi$ are Weierstrass points of $X.$

\section{Holomorphic 1-forms}
\label{form}

In this section our goal is to find holomorphic 1-forms. We will do so by giving $X$ a translational structure that is compatible with its hyperbolic structure. One can define such a structure on a polyhedral surface by using the notion of cone points which we will mention soon. First by Riemann mapping theorem, we can map the hyperbolic triangle $\triangle \, q_1 \, q_2 \, q_3$ to a euclidean triangle $\triangle \, p_1 \, p_2 \, p_3$ via a conformal map. Then by Schwarz reflection principle we can map the rest of the hyperbolic triangles to euclidean ones as shows in Figure~\ref{flat}. Then on this tiling we can find translation structures which are geometric representations of holomorphic 1-forms on $X.$ Once we find a basis of holomorphic 1-forms we can find an algebraic equation that represents this surface.\\

\begin{figure}[htbp] 
   \centering
   \includegraphics[width=4in]{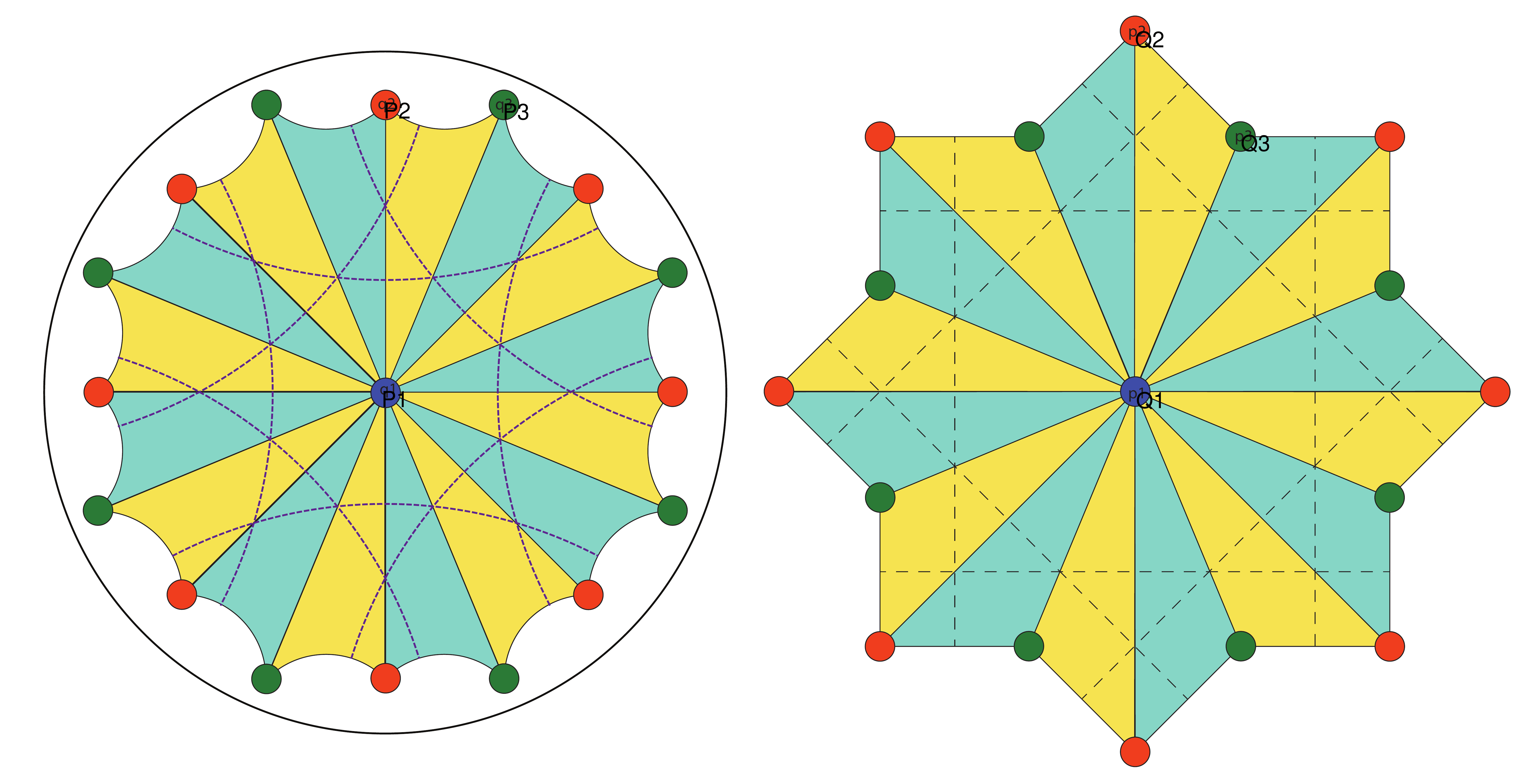} 
   \caption{}
   \label{flat}
\end{figure}

Following the identification of edges in the hyperbolic representation of $X,$ we get a flat 16-gon on the right where every pair of parallel edges are identified. Hence we get a translation structure defined everywhere except at the vertices of the 16-gon. We can put a canonical translation structure $\textrm{d} z$ everywhere except the vertices. Along edges we have $z \mapsto z + c$ for each pair of edges, hence we still have $\textrm{d} z.$ However if we go around say the vertex $p_3$ by the identification of edges, the cone angle is greater than $2 \pi$ and hence the vertex becomes a cone point. The cone angle $\frac{5 \pi}{4} \times 8 = 10 \pi$ shows that in local coordinates we have charts behaving as $z^5.$ We take the exterior derivative and get $5 z^4 \textrm{d} z,$ which is a holomorphic 1-form that has a zero of order four at that point.\\

Here we get a holomophic 1-form that has zeros of order $0, 1, 4$ at the three branch points $p_1, p_2, p_3$ respectively, which corresponds to a holomorphic function that has zeros of order $1, 2, 5$ at $p_1, p_2, p_3$ respectively. Since $X$ is a cyclically branched cover we can take its multipliers and achieve maps that have order $2, 4, 2$ and $5, 2, 1$ at $p_1, p_2, p$ respectively.\\

Now to find an algebraic equation that describes this surface we take the divisors of holomorphic 1-forms which form the basis.

$$\begin{array}{cccc}(\omega_1) = & & p_2 & + 4 p_3\\
(\omega_2) = & p_1 & + 3 p_2 & + p_3\\
(\omega_3) = & 4 p_1 & + p_2 & 
\end{array}$$

We then define holomorphic functions as in \cite{KW}.

$$(f) := \left(\frac{\omega_1}{\omega_3}\right) = -4 p_1 + 4 p_3, \qquad (g) := \left(\frac{\omega_2}{\omega_3}\right) = -3 p_1 + 2 p_2 + p_3.$$

$$(f^2) = -8 p_1 + 8 p_3, \qquad \left(\frac{g^4}{f}\right) = -8 p_1 + 8 p_2.$$

After proper scaling of the functions we get

$$\begin{array}{ll}f^2 -1 = \frac{g^4}{f} & \Leftrightarrow f^3 - f = g^4\\
& \Rightarrow \omega_1^3 \omega_3 - \omega_1 \omega_3^3 = \omega_2^4.
\end{array}$$

\section{Weierstrass points}
\label{Weierstrass}

As an application of Riemann-Roch theorem, we can find Weierstrass points of a surface from the basis of 1-forms. At a generic point on a compact Riemann surface of genus $g,$ there is a basis of forms that each have zeros of order $0, 1, \ldots , g - 1$ resp. at that point. If not, there is a gap in this sequence. As in our case, we have zeros of order 0, 1, 4. We define the weight of a point $\textrm{wt}_p$ by finding the difference of the two sequences. So in our case the weight of a cone point is $(0 - 0) + (1 - 1) + (4 - 2) = 2.$ Points with positive weight are called Weierstrass points.\\

Weierstrass points carry information on the automorphisms of a Riemann surface in a way that all automorphisms preserve Weierstrass points and their weights. In section~\ref{automorphisms} we found all flag-transitive automorphisms and therefore were able to conclude that we found all automorphisms on $X.$

\begin{rem} A compact Riemann surface of genus $g \geq 2$ is hyperelliptic if and only if the weight of every Weierstrass point is $(0 - 0) + (2 - 1) + \cdots + \left( (2 g - 2) - (g - 1)\right) = \frac{(g - 1) g}{2}.$ Since we have a Weierstrass point on $X$ that has weight 2, our surface is not hyperelliptic.
\end{rem}

We will use the following theorem from \cite{FK} to prove the main theorem.\\

\begin{thm*} On a compact Riemann surface of genus $g \geq 1,$ there are finitely many Weierstrass points. Moreover, the sum of weights of the Weierstrass points is $(g - 1) g (g + 1).$
\end{thm*}

\begin{thm} [Main Theorem] All vertices on $\Pi$ are Weierstrass points.
\end{thm}

\begin{proof} Since our surface has genus $g = 3,$ the sum of weights over all points of $X$ add up to $2 \cdot 3 \cdot 4 = 24.$ So far, we have found one point with weight $\textrm{wt}_p = 2.$ However in the previous section, we have proved that the vertices are similar hence the rest also have equal weight. Since we have 12 vertices in total, the weights add up to 24.\\

This proves that all vertices of $\Pi$ are Weierstrass points and that there are no other Weierstrass points. 
\end{proof}

\begin{rem} In \cite{M} it is proved that all minimal surfaces of genus $g = 3$ are hyperelliptic. Despite $X$ having such a large automorphism group, it cannot be a polyhedralization of any minimal surface due to the fact that it is not hyperelliptic. 
\end{rem}


\begin{thebibliography}{10}
\bibliographystyle{abbrv}
\bibitem{C} H. S. M. Coxeter, ``Regular skew polyhedra in three and four dimensions, and their topological analogues'', Proceedings London Mathematical Society, 1938.
\bibitem{CM} H. S. M. Coxeter, W. O. J. Moser, ``Generators and relations for discrete groups'', Springer-Verlag, 4th edition, 1980.  
\bibitem{FK} H. Farkas, I. Kra, ``Riemann Surfaces'', Springer-Verlag, 2nd edition, 1992.
\bibitem{J} G. A. Jones, ``Elementary abelian regular coverings of Platonic maps'', Journal of Algebraic Combinatorics, 2015.
\bibitem{KW} H. Karcher, M. Weber, ``On Klein's Riemann Surface'', \emph{The Eightfold Way}, MSRI Publications, Vol.35, pp.9-49, 1998.
\bibitem{M} W. Meeks, ``The theory of triply periodic minimal surfaces'', Indiana University Math Journal, 1990. 
\end{thebibliography}
\end{document}